\documentclass[12pt]{amsart}
\date{\today}

\input xymatrix
\xyoption{all}

\usepackage{amsmath,amsthm,amsfonts,amssymb,mathrsfs}

\usepackage{color}

\usepackage{hyperref}

\newtheorem{theorem}{Theorem}[section]
\newtheorem{proposition}[theorem]{Proposition}
\newtheorem{corollary}[theorem]{Corollary}
\newtheorem{lemma}[theorem]{Lemma}

\theoremstyle{definition}
\newtheorem{remark}[theorem]{Remark}

\begin{document}

\title[On the monoid of  cofinite partial isometries of $\mathbb{N}^n$ with the usual metric]{On the monoid of  cofinite partial isometries of $\mathbb{N}^n$ with the usual metric}

\author[O.~Gutik and A.~Savchuk]{Oleg~Gutik and Anatolii~Savchuk}
\address{Faculty of Mathematics, Ivan Franko National
University of Lviv, Universytetska 1, Lviv, 79000, Ukraine}
\email{oleg.gutik@lnu.edu.ua, ovgutik@yahoo.com, asavchuk3333@gmail.com}



\begin{abstract}
In this paper we study the structure of the monoid $\mathbf{I}\mathbb{N}_{\infty}^n$ of  cofinite partial isometries of the $n$-th power of the set of positive integers $\mathbb{N}$ with the usual metric for a positive integer $n\geqslant 2$. We describe the elements of the monoid $\mathbf{I}\mathbb{N}_{\infty}^n$ as partial transformation of $\mathbb{N}^n$, the group of units and the subset of idempotents of the semigroup $\mathbf{I}\mathbb{N}_{\infty}^n$, the natural partial order and Green's relations on $\mathbf{I}\mathbb{N}_{\infty}^n$. In particular we show that the quotient semigroup $\mathbf{I}\mathbb{N}_{\infty}^n/\mathfrak{C}_{\textsf{mg}}$, where $\mathfrak{C}_{\textsf{mg}}$ is the minimum group congruence on $\mathbf{I}\mathbb{N}_{\infty}^n$, is isomorphic to the symmetric group $\mathscr{S}_n$ and $\mathscr{D}=\mathscr{J}$ in $\mathbf{I}\mathbb{N}_{\infty}^n$. Also, we prove that for any integer $n\geqslant 2$ the semigroup $\mathbf{I}\mathbb{N}_{\infty}^n$ is isomorphic to the semidirect product  ${\mathscr{S}_n\ltimes_\mathfrak{h}(\mathscr{P}_{\infty}(\mathbb{N}^n),\cup)}$ of the free semilattice with the unit  $(\mathscr{P}_{\infty}(\mathbb{N}^n),\cup)$ by the symmetric group $\mathscr{S}_n$.
\end{abstract}

\maketitle


\section{Introduction and preliminaries}
In this paper we shall follow the terminology of \cite{Clifford-Preston-1961-1967, Lawson-1998}.
We shall denote the  the cardinality of the set $A$ by $|A|$. For any positive integer $n$ by $\mathscr{S}_n$ we denote the group of permutations of the set $\{1,\ldots,n\}$.

A semigroup $S$ is called {\it inverse} if for any
element $x\in S$ there exists a unique $x^{-1}\in S$ such that
$xx^{-1}x=x$ and $x^{-1}xx^{-1}=x^{-1}$. The element $x^{-1}$ is
called the {\it inverse of} $x\in S$. If $S$ is an inverse
semigroup, then the function $\operatorname{inv}\colon S\to S$
which assigns to every element $x$ of $S$ its inverse element
$x^{-1}$ is called an {\it inversion}.

If $S$ is a semigroup, then we shall denote the subset of all
idempotents in $S$ by $E(S)$. If $S$ is an inverse semigroup, then
$E(S)$ is closed under multiplication and we shall refer to $E(S)$ a as
\emph{band} (or the \emph{band of} $S$). Then the semigroup
operation on $S$ determines the following partial order $\preccurlyeq$
on $E(S)$: $e\preccurlyeq f$ if and only if $ef=fe=e$. This order is
called the {\em natural partial order} on $E(S)$. A
\emph{semilattice} is a commutative semigroup of idempotents.

A linearly ordered subset of a poset is called a \emph{chain}. An \emph{$\omega$-chain} is a chain which is order isomorphic to the set of all negative integers with the usual order $\leq$.

If $S$ is an inverse semigroup then the semigroup operation on $S$ determines the following partial order $\preccurlyeq$
on $S$: $s\preccurlyeq t$ if and only if there exists $e\in E(S)$ such that $s=te$. This order is
called the {\em natural partial order} on $S$ \cite{Wagner-1952}.

By
$(\mathscr{P}_{\infty}(X),\cup)$ we shall denote the
\emph{free semilattice with identity} over a set $X$ of cardinality $\geqslant\omega$, i.e.,
$(\mathscr{P}_{\infty}(X),\cup)$ is the set of all finite
subsets (with the empty set) of $X$ with the semilattice
operation ``union''.

A congruence $\mathfrak{C}$ on a semigroup $S$ is called
\emph{non-trivial} if $\mathfrak{C}$ is distinct from universal and
identity congruences on $S$, and a \emph{group congruence} if the quotient
semigroup $S/\mathfrak{C}$ is a group. If $\mathfrak{C}$ is a congruence on a semigroup $S$ then by $\mathfrak{C}^{\sharp}$ we denote the natural homomorphism from $S$ onto the quotient semigroup $S/\mathfrak{C}$. Every inverse semigroup $S$ admits a least (minimum) group
congruence $\mathfrak{C}_{\mathbf{mg}}$:
\begin{equation*}
    a\mathfrak{C}_{\mathbf{mg}} b \; \hbox{ if and only if there exists }\;
    e\in E(S) \; \hbox{ such that }\; ae=be
\end{equation*}
(see \cite[Lemma~III.5.2]{Petrich-1984}).

If $S$ is a semigroup, then we shall denote the Green relations on $S$ by $\mathscr{R}$, $\mathscr{L}$, $\mathscr{J}$, $\mathscr{D}$ and $\mathscr{H}$ (see \cite{Green-1951} or \cite[Section~2.1]{Clifford-Preston-1961-1967}):
\begin{align*}
    &\qquad a\mathscr{R}b \mbox{ if and only if } aS^1=bS^1;\\
    &\qquad a\mathscr{L}b \mbox{ if and only if } S^1a=S^1b;\\
    &\qquad a\mathscr{J}b \mbox{ if and only if } S^1aS^1=S^1bS^1;\\
    &\qquad \mathscr{D}=\mathscr{L}\circ\mathscr{R}=
          \mathscr{R}\circ\mathscr{L};\\
    &\qquad \mathscr{H}=\mathscr{L}\cap\mathscr{R}.
\end{align*}

If $\alpha\colon X\rightharpoonup Y$ is a partial map, then we shall denote
the domain and the range of $\alpha$ by $\operatorname{dom}\alpha$ and $\operatorname{ran}\alpha$, respectively. A partial map $\alpha\colon X\rightharpoonup Y$ is called \emph{cofinite} if both sets $X\setminus\operatorname{dom}\alpha$ and $Y\setminus\operatorname{ran}\alpha$ are finite.

Let $\mathscr{I}_\lambda$ denote the set of all partial one-to-one
transformations of a non-zero  cardinal $\lambda$ together
with the following semigroup operation:
\begin{equation*}
x(\alpha\beta)=(x\alpha)\beta \quad \hbox{if} \quad x\in\operatorname{dom}(\alpha\beta)=\{
y\in\operatorname{dom}\alpha\colon y\alpha\in\operatorname{dom}\beta\}, \qquad  \hbox{for} \quad
\alpha,\beta\in\mathscr{I}_\lambda.
\end{equation*}
 The semigroup
$\mathscr{I}_\lambda$ is called the \emph{symmetric inverse} (\emph{monoid})
\emph{semigroup} over cardinal $\lambda$~(see \cite{Clifford-Preston-1961-1967}). The symmetric inverse
semigroup was introduced by Wagner~\cite{Wagner-1952} and it plays
a major role in the theory of semigroups. By $\mathscr{I}^{\mathrm{cf}}_\lambda$ is denoted a
subsemigroup of injective partial selfmaps of $\lambda$ with
cofinite domains and ranges in $\mathscr{I}_\lambda$. Obviously, $\mathscr{I}^{\mathrm{cf}}_\lambda$ is an inverse
submonoid of the semigroup $\mathscr{I}_\lambda$. The
semigroup $\mathscr{I}^{\mathrm{cf}}_\lambda$  is called the \emph{monoid of
injective partial cofinite selfmaps} of $\lambda$ \cite{Gutik-Repovs-2015}.

\smallskip

A partial transformation $\alpha\colon (X,d)\rightharpoonup (X,d)$ of a metric space $(X,d)$ is called \emph{isometric} or a \emph{partial isometry}, if $d(x\alpha,y\alpha)=d(x,y)$ for all $x,y\in \operatorname{dom}\alpha$. It is obvious that the composition of two partial isometries of a metric space $(X,d)$ is a partial isometry, and the converse partial map to a partial isometry is a partial isometry, too. Hence the set of partial isometries of a metric space $(X,d)$ with the operation the composition of partial isometries is an inverse submonoid of the symmetric inverse monoid over the cardinal $|X|$. Also, it is obvious that the set of partial cofinite isometries of a metric space $(X,d)$ with the operation the composition of partial isometries is an inverse submonoid of the monoid of injective partial cofinite selfmaps of the cardinal $|X|$.

\smallskip

The semigroup $\mathbf{ID}_{\infty}$ of all partial cofinite isometries of the set of integers $\mathbb{Z}$ with the usual metric $d(n,m)=|n-m|$, $n,m\in \mathbb{Z}$ established in the Bezushchak papers \cite{Bezushchak-2004, Bezushchak-2008}. In \cite{Bezushchak-2004}  the generators of the semigroup $\mathbf{ID}_{\infty}$ are described and there proved that $\mathbf{ID}_{\infty}$ has the exponential growth. We remark that the semigroup $\mathbf{ID}_{\infty}$ is inverse submonoid of the  monoid of all partial cofinite bijections of $\mathbb{Z}$, and  elements of $\mathbf{ID}_{\infty}$ are restrictions of isometries of $\mathbb{Z}$ onto its cofinite subsets in the Lawson interpretation (see \cite[p.~9]{Lawson-1998}). Green's relations and principal ideals of $\mathbf{ID}_{\infty}$ are described in \cite{Bezushchak-2008}. In \cite{Gutik-Savchuk-2017} is shown that the quotient semigroup $\mathbf{ID}_{\infty}/\mathfrak{C}_{\mathbf{mg}}$ is isomorphic to the group ${\mathbf{Iso}}(\mathbb{Z})$ of all isometries of $\mathbb{Z}$, the semigroup $\mathbf{ID}_{\infty}$ is $F$-inverse, and $\mathbf{ID}_{\infty}$ is isomorphic to the semidirect product ${\mathbf{Iso}}(\mathbb{Z})\ltimes_\mathfrak{h}\mathscr{P}_{\!\infty}(\mathbb{Z})$ of the free semilattice $(\mathscr{P}_{\!\infty}(\mathbb{Z}),\cup)$ by the group ${\mathbf{Iso}}(\mathbb{Z})$. Also in \cite{Gutik-Savchuk-2017} established semigroup and shift-continuous topologies on $\mathbf{ID}_{\infty}$ and embedding of the discrete semigroup $\mathbf{ID}_{\infty}$ into compact-like topological semigroups.

\smallskip

Let $\mathbf{I}\mathbb{N}_{\infty}$ be the set of all partial cofinite isometries of the set of positive integers $\mathbb{N}$ with the usual metric $d(n,m)=|n-m|$, $n,m\in \mathbb{N}$. Then $\mathbf{I}\mathbb{N}_{\infty}$ with the operation of composition of partial isometries is an inverse submonoid of $\mathscr{I}_\omega$. The semigroup $\mathbf{I}\mathbb{N}_{\infty}$ of all partial co-finite isometries of positive integers is studied in \cite{Gutik-Savchuk-2018}. There we describe the Green relations on the semigroup $\mathbf{I}\mathbb{N}_{\infty}$, its band and proved that $\mathbf{I}\mathbb{N}_{\infty}$ is a simple $E$-unitary $F$-inverse semigroup. Also in \cite{Gutik-Savchuk-2018}, the least group congruence $\mathfrak{C}_{\mathbf{mg}}$ on $\mathbf{I}\mathbb{N}_{\infty}$ is described and proved that the quotient-semigroup  $\mathbf{I}\mathbb{N}_{\infty}/\mathfrak{C}_{\mathbf{mg}}$ is isomorphic to the additive group of integers $\mathbb{Z}(+)$. An example of a non-group congruence on the semigroup $\mathbf{I}\mathbb{N}_{\infty}$ is presented. Also we proved that a congruence on the semigroup $\mathbf{I}\mathbb{N}_{\infty}$ is group if and only if its restriction onto an isomorphic  copy of the bicyclic semigroup in $\mathbf{I}\mathbb{N}_{\infty}$ is a group congruence.

\smallskip

For an arbitrary positive integer $n\geqslant 2$ by $\mathbb{N}^n$ we denote the $n$-th power of the set of positive inters $\mathbb{N}$ with the usual metric:
\begin{equation*}
d((x_1,\cdots,x_n),(y_1,\ldots,y_2))=\sqrt{(x_1-y_1)^2+\cdots+(x_n-y_n)^2}.
\end{equation*}
Let $\mathbf{I}\mathbb{N}_{\infty}^n$ be the set of all partial cofinite isometries of $\mathbb{N}^n$. It is obvious that $\mathbf{I}\mathbb{N}_{\infty}^n$ with the operation of composition of partial isometries is an inverse submonoid of $\mathscr{I}_\omega$ and later by $\mathbf{I}\mathbb{N}_{\infty}^n$ we shall denote the monoid of all partial cofinite isometries of $\mathbb{N}^n$.

\smallskip

By $\mathbb{I}$ we denote the identity map of $\mathbb{N}^n$ which obviously is the unit of the semigroup $\mathbf{I}\mathbb{N}_{\infty}^n$. Later by $H(\mathbb{I})$ we shall denote the group of units  of $\mathbf{I}\mathbb{N}_{\infty}^n$.

\smallskip

In the paper we study the structure of the monoid $\mathbf{I}\mathbb{N}_{\infty}^n$. We describe the elements of the monoid $\mathbf{I}\mathbb{N}_{\infty}^n$ as partial transformations of $\mathbb{N}^n$, the group of units and the subset of idempotents of $\mathbf{I}\mathbb{N}_{\infty}^n$, the natural partial order and Green's relations on $\mathbf{I}\mathbb{N}_{\infty}^n$. In particular we show that the quotient semigroup $\mathbf{I}\mathbb{N}_{\infty}^n/\mathfrak{C}_{\textsf{mg}}$ is isomorphic to the symmetric group $\mathscr{S}_n$ and $\mathscr{D}=\mathscr{J}$ in $\mathbf{I}\mathbb{N}_{\infty}^n$. Also, we prove that for any integer $n\geqslant 2$ the semigroup $\mathbf{I}\mathbb{N}_{\infty}^n$ is isomorphic to the semidirect product ${\mathscr{S}_n\ltimes_\mathfrak{h}(\mathscr{P}_{\infty}(\mathbb{N}^n),\cup)}$ of free semilattice with the unit  $(\mathscr{P}_{\infty}(\mathbb{N}^n),\cup)$ by the symmetric group $\mathscr{S}_n$.

\section{Properties of partial cofinite isometries of $\mathbb{N}^n$}\label{section-2}

The definition of an isometry implies the following proposition.

\begin{proposition}\label{proposition-2.1}
Let $n$ be a positive integer $\geqslant 2$.
Then every permutation $\sigma$ of the set $\{1,\dots,n\}$ induces an isometry $\alpha_\sigma\colon\mathbb N^n\to\mathbb N^n$, $\alpha_\sigma\colon x\mapsto x\circ\alpha$, of the set $\mathbb{N}^n$.
\end{proposition}

We denote $\mathbf{1}=(1,\ldots,1)$ and $\vec{\mathbf{1}}_i=\{(1,\ldots,\underbrace{x_i}_{i\hbox{\footnotesize{-th}}},\ldots,1)\in \mathbb{N}^n\colon x_i\in\mathbb{N}\}$ for any $i\in\{1,\ldots,n\}$. Also for any $\mathbf{x}=(x_1,\ldots,x_n)\in\mathbb{N}^n$ and any positive real number $r$ we put
\begin{equation*}
  \mathcal{M}_r(\mathbf{x})=\left|\left\{\mathbf{y}\in\mathbb{N}^n \colon d(\mathbf{x},\mathbf{y})=r\right\}\right|.
\end{equation*}

\begin{lemma}\label{lemma-2.2}
Let $n$ be a positive integer $\geqslant 2$. Then the following statements holds.
\begin{itemize}
  \item[$(i)$] $\mathcal{M}_1(\mathbf{1})=n$.
  \item[$(ii)$] If for $\mathbf{x}\in \mathbb{N}^n$  exactly $k$ coordinates of $\mathbf{x}$ distinct from $1$ for some $k=1,\ldots,n$, then $\mathcal{M}_1(\mathbf{x})=n+k$.
\end{itemize}
\end{lemma}

\begin{proof}
$(i)$ It is obvious that the set $\left\{\mathbf{y}\in\mathbb{N}^n \colon d(\mathbf{1},\mathbf{y})=1\right\}$ consists of
\begin{equation*}
  \{(\underbrace{2}_{1\hbox{\footnotesize{-st}}},1,\ldots,1), (1,\underbrace{2}_{2\hbox{\footnotesize{-nd}}},\ldots,1), \ldots, (1,\ldots,\underbrace{2}_{n\hbox{\footnotesize{-th}}})\},
\end{equation*}
which implies the equality $\mathcal{M}_1(\mathbf{1})=n$.

\smallskip

$(ii)$ Fix an arbitrary $\mathbf{x}\in \mathbb{N}^n$ such that exactly $k$ coordinates of $\mathbf{x}$ distinct from $1$ for some $k\in\{1,\ldots,n\}$. By Proposition~\ref{proposition-2.1} without loss of generality we may assume that only the first $k$ coordinates of $\mathbf{x}$ distinct from $1$, i.e., $\mathbf{x}=(x_1,\ldots,x_k,1,\ldots,1)$ for some $x_1>1,\ldots,x_k>1$. This implies that the set
$\left\{\mathbf{y}\in\mathbb{N}^n \colon d(\mathbf{x},\mathbf{y})=1\right\}$ coincides with the following set
\begin{align*}
  \{&(x_1-1,\ldots,x_k,1,\ldots,1), (x_1+1,\ldots,x_k,1,\ldots,1), \ldots, \\
  &(x_1,\ldots,x_k-1,1,\ldots,1), (x_1,\ldots,x_k+1,1,\ldots,1),  \\
  & (x_1,\ldots,x_k,\underbrace{2}_{(k+1)\hbox{\footnotesize{-th}}},\ldots,1), \ldots, (x_1,\ldots,x_k,1,\ldots,\underbrace{2}_{n\hbox{\footnotesize{-th}}})\},
\end{align*}
and hence $\mathcal{M}_1(\mathbf{x})=n+k$.
\end{proof}

Lemma~\ref{lemma-2.2}$(i)$ implies the following corollary.

\begin{corollary}\label{corollarry-2.3}
Let $n$ be a positive integer $\geqslant 2$. Then
$(\mathbf{1})\alpha=\mathbf{1}$ for every isometry $\alpha$ of $\mathbb{N}^n$.
\end{corollary}

For any positive integers $k>1$ and $m$ we denote
\begin{equation*}
  \mathbf{k}_1=(\underbrace{k}_{1\hbox{\footnotesize{-st}}},1,\ldots,1),\; \ldots, \; \mathbf{k}_j=(1,\ldots,1,\underbrace{k}_{j\hbox{\footnotesize{-th}}},1,\ldots,1), \; \ldots, \;
  \mathbf{k}_n=(1,\ldots,1,\underbrace{k}_{n\hbox{\footnotesize{-th}}})
\end{equation*}
and
\begin{equation*}
  \mathbf{C}_m=\left\{\left(x_1,\ldots,x_n\right)\in\mathbb N^n \colon x_1\leqslant m,\ldots,x_n\leqslant m\right\}.
\end{equation*}

Later we need the following technical lemma.

\begin{lemma}\label{lemma-2.5}
Let $x$ and $y$ be any positive integers. If the sequence $\{a_i\}_{i\in\mathbb{N}}$, where $a_i=\sqrt{(x+i)^2+y}$, contains an integer then $\{a_i\}_{i\in\mathbb{N}}$ has infinitely many non-integer members.
\end{lemma}

\begin{proof}
Suppose that $n=\sqrt{x^2+y}$ is an integer. Since $x<\sqrt{x^2+y}$ we get that
\begin{align*}
  &2x<2\sqrt{x^2+y} \qquad \Longleftrightarrow\\
  \Longleftrightarrow \qquad & x^2+2x+1+y<x^2+y+2\sqrt{x^2+y}+1 \qquad \Longleftrightarrow\\
  \Longleftrightarrow \qquad & (x+1)^2+y<(\sqrt{x^2+y}+1)^2 \qquad \Longleftrightarrow\\
  \Longleftrightarrow \qquad & \sqrt{(x+1)^2+y}<n+1,
\end{align*}
which implies that $n<\sqrt{(x+1)^2+y}<n+1$. Thus, $\sqrt{(x+1)^2+y}$ is irrational, and hence the statement of the lemma holds.
\end{proof}

\begin{lemma}\label{lemma-2.6}
Let $n$ be a positive integer $\geqslant 2$.
Let $\alpha\colon \mathbb N^n\rightharpoonup\mathbb N^n$ be a partial cofinite isometry of $\mathbb N^n$. Then for every $i\in\{1,\ldots,n\}$ there exists  a unique  $j(i)\in\{1,\ldots,n\}$ and an integer $q(i)$ such that $(\mathbf{m}_i)\alpha=\mathbf{(m+q(i))}_{j(i)}$ for any $\mathbf{m}_i\in\operatorname{dom}\alpha$. Moreover $\alpha$ determines the permutation $i\mapsto j(i)$ of the set $\{1,\ldots,n\}$.
\end{lemma}

\begin{proof}
Since $\alpha$ is a partial cofinite isometry of $\mathbb N^n$ there exists a positive integer $m\geqslant 3$ such that $\mathbb N^n\setminus\operatorname{dom}\alpha\subseteq \mathbf{C}_{m-1}$ and $\mathbb N^n\setminus\operatorname{ran}\alpha\subseteq \mathbf{C}_{m-1}$. Since the set $\mathbf{C}_{m+1}$ is finite there exists a positive integer $k$ such that $\mathbf{p}_i\notin \mathbf{C}_{m+1}$ and $(\mathbf{p}_i)\alpha\notin \mathbf{C}_{m+1}$ for any integer $p\geqslant k$.  By Proposition~\ref{proposition-2.1} without loss of generality we may assume that $i=j(i)=1$, i.e.,   $(\mathbf{k}_1)\alpha=\mathbf{q}_1$ for some positive integer $q$.

\smallskip

We claim that $(\mathbf{(k+s)}_1)\alpha=\mathbf{(q+s)}_1$ for any positive integer $s$. Since $\mathbf{k}_1\notin \mathbf{C}_{m+1}$, we have that
\begin{equation*}
  \left|\left\{\mathbf{x}\in\operatorname{dom}\alpha\colon d(\mathbf{x},\mathbf{(k-2+t)}_1)=1\right\}\right|=n+1,
\end{equation*}
for every positive integer $t$.
Since $\alpha$ is a partial cofinite isometry of $\mathbb N^n$ and $\mathbf{(k-1)}_1\in \operatorname{dom}\alpha$, the above implies that either $(\mathbf{(k+1)}_1)\alpha=\mathbf{(q-1)}_1$ or $(\mathbf{(k+1)}_1)\alpha=\mathbf{(q+1)}_1$. We claim that $(\mathbf{(k+1)}_1)\alpha=\mathbf{(q+1)}_1$. Suppose to the contrary that $(\mathbf{(k+1)}_1)\alpha=\mathbf{(q-1)}_1$.
Next we consider the following sequence $\left\{\mathbf{(k+p)}_1\right\}_{p\in\mathbb{N}}$ in $\mathbb N^n$.
By the above assumption  there exists a positive integer $s$ such that $(\mathbf{(k+p)}_1)\alpha\notin \mathbf{C}_{m+1}\cup\vec{\mathbf{1}}_1\cup\cdots\cup\vec{\mathbf{1}}_n$ for all $p\geqslant s$.  By Lemma~\ref{lemma-2.5} the sequence $\left\{d((\mathbf{(k+p)}_1)\alpha,(\mathbf{k}_1))\alpha\right\}_{p\in\mathbb{N}}$ contains non-integer numbers. But the sequence $\left\{d(\mathbf{(k+p)}_1,\mathbf{k}_1)\right\}_{p\in\mathbb{N}}$ has only positive integers, which contradicts that $\alpha$ is a partial isometry of $\mathbb N^n$. The obtained contradiction implies that $(\mathbf{(k+1)}_1)\alpha=\mathbf{(q+1)}_1$. Next, similar arguments and induction imply that $(\mathbf{(k+p)}_1)\alpha=\mathbf{(q+p)}_1$ for any $p\in\mathbb{N}$.

\smallskip

Fix an arbitrary integer $l<k$ such that $\mathbf{l}_1\in\operatorname{dom}\alpha$. We claim that $(\mathbf{l}_1)\alpha=(\mathbf{q-k+l})_1$. Indeed, in the other case we have that $d(\mathbf{l}_1,\mathbf{(k+p)}_1)=k+p-l$ for any $p\in\mathbb{N}$, and hence the sequence $\left\{d(\mathbf{l}_1,\mathbf{(k+p)}_1)\right\}_{p\in\mathbb{N}}$ contains only integer elements. But since $(\mathbf{l}_1)\alpha=(x_1,x_2,\ldots,x_n)$ with some $x_i\neq 1$, $i=2,\ldots,n$, we have that
\begin{equation*}
  d((\mathbf{l}_1)\alpha,(\mathbf{(k+p)}_1)\alpha)=\sqrt{(x_1-k+p)^2+(x_2-1)^2+\cdots+(x_2-1)^n},
\end{equation*}
and hence by Lemma~\ref{lemma-2.5} the sequence $\left\{d((\mathbf{l}_1)\alpha,(\mathbf{(k+p)}_1)\alpha)\right\}_{p\in\mathbb{N}}$ contains non-integer elements. This contradicts that $\alpha$ is a partial isometry of $\mathbb{N}^n$, which completes the proof of the lemma.
\end{proof}

\begin{lemma}\label{lemma-2.8}
Let $n$ be a positive integer $\geqslant 2$.
Let $\alpha\colon \mathbb N^n\rightharpoonup\mathbb N^n$ be a partial cofinite isometry of $\mathbb N^n$ such that $(\mathbf{m}_i)\alpha\in\vec{\mathbf{1}}_i$ for some $\mathbf{m}_i\in\operatorname{dom}\alpha\setminus\{\mathbf{1}\}$, $m\in\mathbb{N}$, and any $i\in\{1,\ldots,n\}$. Then $(\mathbf{x}_i)\alpha=\mathbf{x}_i$ for any $\mathbf{x}_i\in\operatorname{dom}\alpha\cap\vec{\mathbf{1}}_i$.
\end{lemma}

\begin{proof}
Suppose to the contrary that there exists $i\in\{1,\ldots,n\}$ such that $(\mathbf{x}_i)\alpha\neq\mathbf{x}_i$ for some $\mathbf{x}_i\in \vec{\mathbf{1}}_i$. Without loss of generality we may assume that $i=1$. Then there exists a non-zero integer $p$ such that $(\mathbf{x}_1)\alpha=\mathbf{(x+p)}_1$. We assume that $p>0$.

\smallskip

Fix an arbitrary $\mathbf{y}_2\in\operatorname{dom}\alpha\cap \vec{\mathbf{1}}_2\setminus \{\mathbf{1}\}$. Lemma~\ref{lemma-2.6} implies that without loss of generality we may assume that $y<x$.  Since $\alpha$ is a partial isometry of $\mathbb{N}^n$,
\begin{equation}\label{eq-2.1}
  d((\mathbf{x}_1)\alpha,(\mathbf{y}_2)\alpha)=d(\mathbf{x}_1,\mathbf{y}_2)=\sqrt{(x-1)^2+(y-1)^2}.
\end{equation}
This and the equality $(\mathbf{x}_1)\alpha=\mathbf{(x+p)}_1$ imply that there exists a positive integer $q$ such that $(\mathbf{y}_2)\alpha=\mathbf{(x-q)}_2$. Then we have that
\begin{equation}\label{eq-2.2}
  d((\mathbf{x}_1)\alpha,(\mathbf{y}_2)\alpha)=d(\mathbf{(x+p)}_1,\mathbf{(y-q)}_2)=\sqrt{(x+p-1)^2+(y-q-1)^2}.
\end{equation}
By equalities \eqref{eq-2.1} and \eqref{eq-2.2} we get that
\begin{equation*}
  (x-1)^2+(y-1)^2=(x-1)^2+2(x-1)p+p^2+(y-1)^2-2(y-1)q+q^2,
\end{equation*}
and hence
\begin{equation}\label{eq-2.3}
  p^2+q^2+2(xp-yq)-2(p-q)=0.
\end{equation}
Again, since $\alpha$ is a partial isometry of $\mathbb{N}^n$, Lemma~\ref{lemma-2.6} implies that
\begin{equation}\label{eq-2.4}
  d((\mathbf{(x+1)}_1)\alpha,(\mathbf{(y+1)}_2)\alpha)=d(\mathbf{(x+1)}_1,\mathbf{(y+1)}_2)=\sqrt{(x+1-1)^2+(y+1-1)^2}=\sqrt{x^2+y^2},
\end{equation}
$(\mathbf{(x+1)}_1)\alpha=\mathbf{(x+1+p)}_1$ and $(\mathbf{(y+1)}_2)\alpha=\mathbf{(y+1-q)}_2$. Then
\begin{equation}\label{eq-2.5}
\begin{split}
  d((\mathbf{(x+1)}_1)\alpha,(\mathbf{(y+1)}_2)\alpha) & =d(\mathbf{(x+1+p)}_1,\mathbf{(y+1-q)}_2)= \\
    & =\sqrt{(x+1+p-1)^2+(y+1-q-1)^2}= \\
    & =\sqrt{(x+p)^2+(y-q)^2}.
\end{split}
\end{equation}
By equalities \eqref{eq-2.4} and \eqref{eq-2.5} we have that
\begin{equation*}
  x^2+y^2=x^2+2xp+p^2+y^2-2yq+q^2,
\end{equation*}
and after some arithmetic simplifications we get that
\begin{equation}\label{eq-2.6}
  p^2+q^2+2(xp-yq)=0.
\end{equation}
Then equalities \eqref{eq-2.3} and \eqref{eq-2.6} imply that $p=q$. Therefore, $(\mathbf{y}_2)\alpha=\mathbf{(x-p)}_2$. Then equalities \eqref{eq-2.1} and \eqref{eq-2.2} imply that
\begin{equation*}
  (x-1)^2+(y-1)^2=(x-1)^2+2(x-1)p+p^2+(y-1)^2-2(y-1)p+p^2,
\end{equation*}
and hence
\begin{equation*}
  2p(x-y+p)=0.
\end{equation*}
This contradicts the assumptions $p>0$ and $x>y$.

\smallskip

In the case when $p<0$ we choose $x<y$, and by similar way as in the above we get a contradiction. The obtained contradictions imply that $(\mathbf{x}_i)\alpha=\mathbf{x}_i$ for any $\mathbf{x}_i\in \vec{\mathbf{1}}_i$ with $x\geqslant k$, $i=1,\ldots,n$. This completes the proof of the lemma.
\end{proof}

\begin{theorem}\label{theorem-2.9}
Let $n$ be a positive integer $\geqslant 2$ and $\alpha\colon \mathbb N^n\rightharpoonup\mathbb N^n$ be a partial cofinite isometry of $\mathbb N^n$ such that $(\mathbf{m}_i)\alpha\in\vec{\mathbf{1}}_i$ for some $\mathbf{m}_i\in\operatorname{dom}\alpha\setminus\{\mathbf{1}\}$, $m\in\mathbb{N}$, and any $i\in\{1,\ldots,n\}$. Then $(\mathbf{x})\alpha=\mathbf{x}$ for any $\mathbf{x}\in\operatorname{dom}\alpha$.
\end{theorem}

\begin{proof}
By Lemma~\ref{lemma-2.8} it is sufficiently to show that $(\mathbf{x})\alpha=\mathbf{x}$ for any $\mathbf{x}\in\operatorname{dom}\alpha\setminus (\vec{\mathbf{1}}_1\cup\cdots \cup\vec{\mathbf{1}}_n)$.

\smallskip

Suppose to the contrary that there exists $\mathbf{x}=(x_1,\ldots,x_n)\in\operatorname{dom}\alpha\setminus (\vec{\mathbf{1}}_1\cup\cdots \cup\vec{\mathbf{1}}_n)$ with $x_1,\ldots,x_n\neq1$ such that $(\mathbf{x})\alpha=(y_1,\ldots,y_n)\neq\mathbf{x}$. Lemma~\ref{lemma-2.6} implies that $(\mathbf{p}_1)\alpha=\mathbf{p}_1$ for all positive integers $p$ such that $\mathbf{p}_1\in\operatorname{dom}\alpha$. Then we have that
\begin{equation*}
  d(\mathbf{x},\mathbf{p}_1)=d((\mathbf{x})\alpha,(\mathbf{p}_1)\alpha) \quad \Longleftrightarrow
\end{equation*}
\begin{equation*}
 (x_1-p)^2+(x_2-1)^2+\cdots+(x_n-1)^2=(y_1-p)^2+(y_2-1)^2+\cdots+(y_n-1)^2  \quad \Longleftrightarrow
\end{equation*}
\begin{equation*}
 x_1^2-2x_1p+p^2+(x_2-1)^2+\cdots+(x_n-1)^2-y_1^2+2y_1p-p^2-(y_2-1)^2-\cdots-(y_n-1)^2=0  \quad \Longleftrightarrow
\end{equation*}
\begin{equation}\label{eq-2.7}
   2(y_1-x_1)p+(x_2-1)^2+\cdots+(x_n-1)^2+x_1^2-y_1^2-(y_2-1)^2-\cdots-(y_n-1)^2=0.
\end{equation}
Since the partial map $\alpha$ is cofinite equation \eqref{eq-2.7} has infinitely many solutions in $\mathbb{N}$ with the respect to the variable $p$. This implies that $y_1=x_1$ and
\begin{equation*}
  (x_2-1)^2+\cdots+(x_n-1)^2+x_1^2-y_1^2-(y_2-1)^2-\cdots-(y_n-1)^2=0.
\end{equation*}
 Similar arguments and Lemma~\ref{lemma-2.6} imply that $y_i=x_i$ for any other $i\in\{2,\ldots,n\}$.
\end{proof}

\section{Algebraic property of the semigroup $\mathbf{I}\mathbb{N}_{\infty}^n$}\label{section-3}

Theorem~\ref{theorem-2.9} implies the following corollary.

\begin{corollary}\label{corollary-3.1}
Let $n$ be any positive integer $\geqslant 2$ and $\alpha\colon \mathbb{N}^n\to \mathbb{N}^n$ be an isometry such that $(\mathbf{2}_i)\alpha=\mathbf{2}_i$ for all $i=1,\ldots,n$. Then $\alpha$ is the identity map of $\mathbb{N}^n$.
\end{corollary}

The following theorem describes the structure of the group of units $H(\mathbb{I})$ of the semigroup $\mathbf{I}\mathbb{N}_{\infty}^n$.

\begin{theorem}\label{theorem-3.2}
For any positive integer $n\geqslant 2$ the group of units $H(\mathbb{I})$ of the semigroup $\mathbf{I}\mathbb{N}_{\infty}^n$ is isomorphic to the group $\mathscr{S}_n$. Moreover, every element of $H(\mathbb{I})$ is induced by a permutation of the set $\{1,\ldots,n\}$.
\end{theorem}

\begin{proof}
It is obvious that every element of the group of units $H(\mathbb{I})$ of the semigroup $\mathbf{I}\mathbb{N}_{\infty}^n$ is an isometry of $\mathbb{N}^n$.

\smallskip

Fix an arbitrary isometry $\alpha$ of $\mathbb{N}^n$. Since the set $\left\{\mathbf{y}\in\mathbb{N}^n \colon d(\mathbf{1},\mathbf{y})=1\right\}$ coincides with
\begin{equation*}
  \mathbf{D}_1(\mathbf{1})=\{(\underbrace{2}_{1\hbox{\footnotesize{-st}}},1,\ldots,1), (1,\underbrace{2}_{2\hbox{\footnotesize{-nd}}},\ldots,1), \ldots, (1,\ldots,\underbrace{2}_{n\hbox{\footnotesize{-th}}})\},
\end{equation*}
we have that $( \mathbf{D}_1(\mathbf{1}))\alpha= \mathbf{D}_1(\mathbf{1})$. Then there exists a permutation $\sigma\colon\{1,\ldots,n\}\to\{1,\ldots,n\}$ such that $(\mathbf{2}_i)\alpha=\mathbf{2}_{(i)\sigma}$ for any $i=1,\ldots,n$. By Proposition~\ref{proposition-2.1} the permutation $\sigma$ induces the isometry $\alpha_\sigma\colon\mathbb N^n\to\mathbb N^n$, $x\mapsto x\circ\alpha$ of the set $\mathbb{N}^n$. It is obvious that $(\mathbf{2}_i)\alpha\alpha_\sigma^{-1}=\mathbf{2}_i$ and $(\mathbf{2}_j)\alpha_\sigma^{-1}\alpha=\mathbf{2}_j$ for all $i,j=1,\ldots,n$. Since $\alpha\alpha_\sigma^{-1}$ and $\alpha_\sigma^{-1}\alpha$ are isometries of $\mathbb{N}^n$, Corollary~\ref{corollary-3.1} implies that $\alpha\alpha_\sigma^{-1}=\alpha_\sigma^{-1}\alpha=\mathbb{I}$. This implies that the isometry $\alpha$ is induced by some permutation of the set $\{1,\ldots,n\}$, which completes the proof of the theorem.
\end{proof}

\begin{proposition}\label{proposition-3.3}
Let $n$ be any positive integer $\geqslant 2$. Then the following statements hold.
\begin{itemize}
    \item[$(i)$] An element $\alpha$ of the semigroup $\mathbf{I}\mathbb{N}_{\infty}^n$ is an idempotent if and only if $(x)\alpha=x$ for every $x\in\operatorname{dom}\alpha$.

    \item[$(ii)$] If $\varepsilon,\iota\in E(\mathbf{I}\mathbb{N}_{\infty}^n)$, then $\varepsilon\leqslant\iota$ if and only if           $\operatorname{dom}\varepsilon\subseteq \operatorname{dom}\iota$.

    \item[$(iii)$] The semilattice $E(\mathbf{I}\mathbb{N}_{\infty}^n)$ is isomorphic to the free semilattice           $(\mathscr{P}_{\infty}(\mathbb{N}^n),\cup)$ under the mapping $(\varepsilon)h=\mathbb{N}\setminus           \operatorname{dom}\varepsilon$.

    \item[$(iv)$] Every maximal chain in $E(\mathbf{I}\mathbb{N}_{\infty}^n)$ is an $\omega$-chain.

    \item[$(v)$] $\mathbf{I}\mathbb{N}_{\infty}^n$ is an inverse semigroup.

    \item[$(vi)$] $\alpha\mathscr{R}\beta$ in $\mathbf{I}\mathbb{N}_{\infty}^n$ if and only if          $\operatorname{dom}\alpha=\operatorname{dom}\beta$.

    \item[$(vii)$] $\alpha\mathscr{L}\beta$ in $\mathbf{I}\mathbb{N}_{\infty}^n$ if and only if          $\operatorname{ran}\alpha=\operatorname{ran}\beta$.

    \item[$(viii)$] $\alpha\mathscr{H}\beta$ in $\mathbf{I}\mathbb{N}_{\infty}^n$ if and only if          $\operatorname{dom}\alpha=\operatorname{dom}\beta$ and $\operatorname{ran}\alpha=\operatorname{ran}\beta$.
\end{itemize}
\end{proposition}

\begin{proof}
Statements $(i)-(iv)$ are trivial and their proofs follow from the definition of the semigroup $\mathbf{I}\mathbb{N}_{\infty}^n$.

\smallskip

$(v)$ Fix an arbitrary $\alpha\in\mathbf{I}\mathbb{N}_{\infty}^n$. Since $\alpha$ is a partial cofinite isometry of $\mathbb{N}^n$, we conclude that so is its inverse $\alpha^{-1}$. This implies that $\mathbf{I}\mathbb{N}_{\infty}^n$ is regular, and hence by statement $(iii)$ and the Wagner-Preston Theorem (see \cite[Theorem~1.17]{Clifford-Preston-1961-1967}), $\mathbf{I}\mathbb{N}_{\infty}^n$ is an inverse semigroup.

\smallskip

Statements $(vi)-(viii)$ follow from the description of Green's relations $\mathscr{R}$, $\mathscr{L}$ and $\mathscr{H}$ on the symmetric inverse monoid $\mathscr{I}_\omega$ and Proposition~3.2.11 of \cite{Lawson-1998}.
\end{proof}

\begin{lemma}\label{lemma-3.4}
Let $n$ be any positive integer $\geqslant 2$ and  $\alpha$ be an arbitrary element of the semigroup $\mathbf{I}\mathbb{N}_{\infty}^n$. Then there exist unique elements $\sigma_l$ and $\sigma_r$ of the group of units $H(\mathbb{I})$ of the semigroup $\mathbf{I}\mathbb{N}_{\infty}^n$ such that $\sigma_l\alpha$ and $\alpha\sigma_r$ are idempotents in $\mathbf{I}\mathbb{N}_{\infty}^n$.
\end{lemma}

\begin{proof}
Suppose that $\sigma\colon \{1,\ldots,n\}\to \{1,\ldots,n\}, \; i\mapsto j(i)$ is a permutation defined in Lemma~\ref{lemma-2.6} for the partial isometry $\alpha$. We put $\sigma_{\alpha}$ is an element of $H(\mathbb{I})$ which is induced by the permutation $\sigma$ of the set $\{1,\ldots,n\}$ (see Theorem~\ref{theorem-3.2}) and $\sigma_{\alpha}^{-1}$ is its inverse. Then the assumption of Theorem~\ref{theorem-2.9} holds for $\alpha\sigma_\alpha^{-1}$, and hence by Proposition~\ref{proposition-3.3}$(i)$, $\alpha\sigma_{\alpha}^{-1}$ is an idempotent of $\mathbf{I}\mathbb{N}_{\infty}^n$. Next we put $\sigma_r=\sigma_{\alpha}^{-1}$.

\smallskip

By a similar way we suppose that $\overline{\sigma}\colon \{1,\ldots,n\}\to \{1,\ldots,n\}, \; i\mapsto j(i)$ is a permutation defined in Lemma~\ref{lemma-2.6} for $\alpha^{-1}$, i.e., $\overline{\sigma}$ is inverse of $\sigma$. We put $\overline{\sigma}_{\alpha}$ is an element of $H(\mathbb{I})$ which is induced by the permutation $\overline{\sigma}$ of the set $\{1,\ldots,n\}$  and $\overline{\sigma}_{\alpha}^{-1}$ is its inverse. Then by Proposition~\ref{proposition-3.3}$(i)$ the element $\alpha^{-1}\overline{\sigma}_{\alpha}^{-1}$ is an idempotent of $\mathbf{I}\mathbb{N}_{\infty}^n$, and hence so is its inverse $(\alpha^{-1}\overline{\sigma}_{\alpha}^{-1})^{-1}=\overline{\sigma}_{\alpha}\alpha$, because by Proposition~\ref{proposition-3.3}$(v)$ the semigroup $\mathbf{I}\mathbb{N}_{\infty}^n$ is inverse. Hence, we put $\sigma_l=\sigma_{\alpha}$.

We observe that the uniqueness of the elements   $\sigma_l$ and $\sigma_r$ follows from Lemma~\ref{lemma-2.6}.
\end{proof}

\begin{lemma}\label{lemma-3.5}
Let $n$ be any positive integer $\geqslant 2$ and $\alpha$ be an arbitrary element of the semigroup $\mathbf{I}\mathbb{N}_{\infty}^n$. Then there exists the unique element $\sigma_{\alpha}$ of the group of units $H(\mathbb{I})$ and the unique idempotents $\varepsilon_{l(\alpha)}$ and $\varepsilon_{r(\alpha)}$ of the semigroup $\mathbf{I}\mathbb{N}_{\infty}^n$ such that
\begin{equation*}
\alpha=\sigma_{\alpha}\varepsilon_{l(\alpha)}=\varepsilon_{r(\alpha)}\sigma_{\alpha}.
\end{equation*}
\end{lemma}

\begin{proof}
By Lemma~\ref{lemma-3.4} there exists unique elements  $\sigma_{l_{\alpha}}$ and $\sigma_{r_{\alpha}}$  of the group of units $H(\mathbb{I})$ such that $\sigma_{l_{\alpha}}\alpha$ and $\alpha\sigma_{r_{\alpha}}$ are idempotents of the semigroup $\mathbf{I}\mathbb{N}_{\infty}^n$. This implies that
\begin{equation}\label{eq-3.1}
\sigma_{l_{\alpha}}\alpha=\alpha^{-1}\alpha \qquad \hbox{and} \qquad \alpha\sigma_{r_{\alpha}}=\alpha\alpha^{-1},
\end{equation}
because $\mathbf{I}\mathbb{N}_{\infty}^n$ is an inverse semigroup, $\operatorname{ran}(\sigma_{l_{\alpha}}\alpha)=\operatorname{ran}(\alpha^{-1}\alpha)$ and $\operatorname{dom}(\alpha\sigma_{r_{\alpha}})=\operatorname{dom}(\alpha\alpha^{-1})$. Hence we get that
\begin{equation*}
\alpha=\sigma_{l_{\alpha}}\sigma_{l_{\alpha}}\alpha=\sigma_{l_{\alpha}}\alpha^{-1}\alpha \qquad \hbox{and} \qquad \alpha=\alpha\sigma_{r_{\alpha}}\sigma_{r_{\alpha}}^{-1}=\alpha\alpha^{-1}\sigma_{r_{\alpha}}^{-1}.
\end{equation*}
We put
\begin{equation*}
  \sigma_{l(\alpha)}=\sigma_{l_{\alpha}}, \qquad \sigma_{r(\alpha)}=\sigma_{r_{\alpha}}^{-1}, \qquad \varepsilon_{l(\alpha)}=\alpha^{-1}\alpha, \qquad \hbox{and} \qquad \varepsilon_{r(\alpha)}=\alpha\alpha^{-1}.
\end{equation*}
We claim that $\sigma_{l(\alpha)}=\sigma_{r(\alpha)}$. Indeed, the inequality $\sigma_{l(\alpha)}\neq\sigma_{r(\alpha)}$ and Theorem~\ref{theorem-3.2} imply that the restriction $\sigma_{l(\alpha)}^{-1}\sigma_{r(\alpha)}|_{\mathbb{N}^n\setminus \mathbf{C}_m}\colon \mathbb{N}^n\setminus \mathbf{C}_m\to \mathbb{N}^n\setminus \mathbf{C}_m$ is not the identity map for any positive integer $m$. This implies that $\varepsilon_{l(\alpha)}\neq\sigma_{l(\alpha)}^{-1}\varepsilon_{r(\alpha)}\sigma_{r(\alpha)}$, which contradicts the equalities $\alpha=\sigma_{l(\alpha)}\varepsilon_{l(\alpha)}=\varepsilon_{r(\alpha)}\sigma_{r(\alpha)}$. Hence $\sigma_{l(\alpha)}=\sigma_{r(\alpha)}$ and we put $\sigma_{\alpha}=\sigma_{r(\alpha)}$.

\smallskip

The uniqueness of the elements $\sigma_{\alpha}$, $\varepsilon_{l(\alpha)}$ and $\varepsilon_{r(\alpha)}$ follows from Lemma~\ref{lemma-2.6} and equalities \eqref{eq-3.1}, because $\mathbf{I}\mathbb{N}_{\infty}^n$ is an inverse semigroup.
\end{proof}

Lemma~\ref{lemma-3.5} implies the following corollary.

\begin{corollary}
For any positive integer $n\geqslant 2$ the subset $H(\mathbb{I})\cup E(\mathbf{I}\mathbb{N}_{\infty}^n)\subset\mathbf{I}\mathbb{N}_{\infty}^n$ is a set of generators of $\mathbf{I}\mathbb{N}_{\infty}^n$. Moreover, so is any subset $A=A_1\sqcup A_2\subset\mathbf{I}\mathbb{N}_{\infty}^n$ such that $A_1$ generates the group $H(\mathbb{I})$ and $A_2$ generates the set $E(\mathbf{I}\mathbb{N}_{\infty}^n)$.
\end{corollary}

The following theorem describes the natural partial order on the semigroups $\mathbf{I}\mathbb{N}_{\infty}^n$ in the terms of Lemma~\ref{lemma-3.5}.

\begin{theorem}\label{theorem-3.6}
Let $n$ be any positive integer $\geqslant 2$.
Let $\alpha$ and $\beta$ be elements of the semigroup $\mathbf{I}\mathbb{N}_{\infty}^n$.
Let
\begin{equation*}
\alpha=\sigma_{\alpha}\varepsilon_{l(\alpha)}=\varepsilon_{r(\alpha)}\sigma_{\alpha} \qquad \hbox{and} \qquad \beta=\sigma_{\beta}\varepsilon_{l(\beta)}=\varepsilon_{r(\beta)}\sigma_{\beta}
\end{equation*}
for some elements $\sigma_{\alpha}$ and $\sigma_{\beta}$ of the group of units $H(\mathbb{I})$ and idempotents $\varepsilon_{l(\alpha)}$, $\varepsilon_{r(\alpha)}$, $\varepsilon_{l(\beta)}$ and $\varepsilon_{r(\beta)}$ of the semigroup $\mathbf{I}\mathbb{N}_{\infty}^n$. Then $\alpha\preccurlyeq\beta$ in $\mathbf{I}\mathbb{N}_{\infty}^n$ if and only if $\sigma_{\alpha}=\sigma_{\beta}$, $\varepsilon_{l(\alpha)}\preccurlyeq\varepsilon_{l(\beta)}$ and $\varepsilon_{r(\alpha)}\preccurlyeq\varepsilon_{r(\beta)}$ in $E(\mathbf{I}\mathbb{N}_{\infty}^n)$.
\end{theorem}

\begin{proof}
$(\Rightarrow)$
Suppose that $\alpha\preccurlyeq\beta$ in $\mathbf{I}\mathbb{N}_{\infty}^n$. We consider the case   $\alpha=\sigma_{\alpha}\varepsilon_{l(\alpha)}$ and $\beta=\sigma_{\beta}\varepsilon_{l(\beta)}$. Then there exists an idempotent $\varepsilon\in \mathbf{I}\mathbb{N}_{\infty}^n$ such that $\alpha=\beta\varepsilon$. Then $\sigma_{\alpha}\varepsilon_{l(\alpha)}=\sigma_{\beta}\varepsilon_{l(\beta)}\varepsilon$ and hence
\begin{equation*}
\varepsilon_{l(\alpha)}=\sigma_{\alpha}^{-1}\sigma_{\alpha}\varepsilon_{l(\alpha)}=
\sigma_{\alpha}^{-1}\sigma_{\beta}\varepsilon_{l(\beta)}\varepsilon,
\end{equation*}
because $\sigma_{\alpha}\in H(\mathbb{I})$. Since $\varepsilon_{l(\alpha)}$, $\varepsilon_{l(\beta)}$ and $\varepsilon$ are idempotents of $\mathbf{I}\mathbb{N}_{\infty}^n$ there exists a positive integer $m$ such that
\begin{equation*}
  \mathbb N^n\setminus\operatorname{dom}\varepsilon_{l(\alpha)}\cup\mathbb N^n\setminus\operatorname{dom}\varepsilon_{l(\beta)}\cup \mathbb N^n\setminus\operatorname{dom}\varepsilon \subseteq \mathbf{C}_{m},
\end{equation*}
and hence by Proposition~\ref{proposition-3.3}, $\mathbb{N}^n\setminus\operatorname{dom}(\varepsilon_{l(\beta)}\varepsilon) \subseteq \mathbf{C}_{m}$. Then the equality $\varepsilon_{l(\alpha)}=\sigma_{\alpha}^{-1}\sigma_{\beta}\varepsilon_{l(\beta)}\varepsilon$ implies that the restriction of $\sigma_{\alpha}^{-1}\sigma_{\beta}|_{\mathbb{N}^n\setminus \mathbf{C}_{m}}\colon \mathbb{N}^n\setminus \mathbf{C}_{m}\rightharpoonup \mathbb{N}^n$ is the identity partial map of $\mathbb{N}^n\setminus \mathbf{C}_{m}$, and hence by Theorem~\ref{theorem-3.2} we get that $\sigma_{\alpha}=\sigma_{\beta}$. Then
\begin{equation*}
\varepsilon_{l(\alpha)}=\sigma_{\alpha}^{-1}\sigma_{\beta}\varepsilon_{l(\beta)}\varepsilon= \sigma_{\alpha}^{-1}\sigma_{\alpha}\varepsilon_{l(\beta)}\varepsilon=\mathbb{I}\varepsilon_{l(\beta)}\varepsilon=\varepsilon_{l(\beta)}\varepsilon,
\end{equation*}
and hence $\varepsilon_{l(\alpha)}\preccurlyeq\varepsilon_{l(\beta)}$ in $E(\mathbf{I}\mathbb{N}_{\infty}^n)$.

\smallskip

The proof in the case $\alpha=\varepsilon_{r(\alpha)}\sigma_{\alpha}$ and $\beta=\varepsilon_{r(\beta)}\sigma_{\beta}$ is similar using Lemma~1.4.6 of \cite{Lawson-1998}.

\smallskip

The implication $(\Leftarrow)$ is trivial and it follows from Lemma~1.4.6 and Proposition~1.4.7 of \cite{Lawson-1998}.
\end{proof}

Since $\alpha\mathfrak{C}_{\textsf{mg}}\beta$ in $\mathbf{I}\mathbb{N}_{\infty}^n$ if and only if there exists $\gamma\in\mathbf{I}\mathbb{N}_{\infty}^n$ such that $\gamma\preccurlyeq \alpha$ and $\gamma\preccurlyeq\beta$ (see \cite[Section~2.4, p.~62]{Lawson-1998}), Theorem~\ref{theorem-3.6} implies that $\alpha\mathfrak{C}_{\textsf{mg}}\beta$ in $\mathbf{I}\mathbb{N}_{\infty}^n$ if and only if $\sigma_{\alpha}=\sigma_{\beta}$. Hence the following theorem holds.

\begin{theorem}\label{theorem-3.7}
Let $n$ be any positive integer $\geqslant 2$. Then the quotient semigroup $\mathbf{I}\mathbb{N}_{\infty}^n/\mathfrak{C}_{\textsf{mg}}$ is isomorphic to the group $\mathscr{S}_n$ and the natural homomorphism $\mathfrak{C}_{\textsf{mg}}^{\sharp}\colon \mathbf{I}\mathbb{N}_{\infty}^n\to \mathbf{I}\mathbb{N}_{\infty}^n/\mathfrak{C}_{\textsf{mg}}$ is defined in the following way: $\alpha\mapsto \sigma_\alpha$.
\end{theorem}

The following theorem gives more detail description of Green's relations $\mathscr{R}$, $\mathscr{L}$, $\mathscr{H}$, $\mathscr{D}$ and $\mathscr{J}$ on the semigroup $\mathbf{I}\mathbb{N}_{\infty}^n$.

\begin{theorem}\label{theorem-3.8}
Let $n$ be any positive integer $\geqslant 2$ and $\alpha,\beta\in\mathbf{I}\mathbb{N}_{\infty}^n$. Then the following statements hold:
\begin{itemize}
    \item[$(i)$] $\alpha\mathscr{L}\beta$ if and only if there exists $\sigma\in H(\mathbb{I})$ such that $\alpha=\sigma\beta$;

    \item[$(ii)$] $\alpha\mathscr{R}\beta$ if and only if there exists $\sigma\in H(\mathbb{I})$ such that $\alpha=\beta\sigma$;

    \item[$(iii)$] $\alpha\mathscr{H}\beta$ if and only if there exist $\sigma_1,\sigma_2\in H(\mathbb{I})$ such that $\alpha=\sigma_1\beta$ and $\alpha=\beta\sigma_2$;

    \item[$(iv)$] $\alpha\mathscr{D}\beta$ if and only if there exist $\sigma_1,\sigma_2\in H(\mathbb{I})$ such that $\alpha=\sigma_1\beta\sigma_2$;

    \item[$(v)$] $\mathscr{D}=\mathscr{J}$ on $\mathbf{I}\mathbb{N}_{\infty}^n$;

    \item[$(vi)$] every $\mathscr{J}$-class in $\mathbf{I}\mathbb{N}_{\infty}^n$ is finite and consists of incomparable elements with the respect to the natural partial order $\preccurlyeq$ on $\mathbf{I}\mathbb{N}_{\infty}^n$.
\end{itemize}
\end{theorem}

\begin{proof}
$(i)$ $(\Rightarrow)$ Suppose that $\alpha\mathscr{L}\beta$ in $\mathbf{I}\mathbb{N}_{\infty}^n$. Then $\alpha^{-1}\alpha=\beta^{-1}\beta$ and by Lemma~\ref{lemma-3.4} there exists unique elements  $\sigma_{l_{\alpha}}$ and $\sigma_{l_{\beta}}$  of the group of units $H(\mathbb{I})$ such that $\sigma_{l_{\alpha}}\alpha$ and $\sigma_{l_{\beta}}\alpha$ are idempotents of the semigroup $\mathbf{I}\mathbb{N}_{\infty}^n$. This implies that
\begin{equation*}
\sigma_{l_{\alpha}}\alpha=\alpha^{-1}\alpha=\beta^{-1}\beta=\sigma_{l_{\beta}}\beta,
\end{equation*}
because $\mathbf{I}\mathbb{N}_{\infty}^n$ is an inverse semigroup and
\begin{equation*}
\operatorname{ran}(\sigma_{l_{\alpha}}\alpha)=\operatorname{ran}(\alpha^{-1}\alpha)=\operatorname{ran}(\beta^{-1}\beta)= \operatorname{ran}(\sigma_{l_{\beta}}\beta).
\end{equation*}
Hence we get that $\alpha=\mathbb{I}\alpha=\sigma_{l_{\alpha}}^{-1}\sigma_{l_{\alpha}}\alpha=\sigma_{l_{\alpha}}^{-1}\sigma_{l_{\beta}}\beta$, and the element $\sigma=\sigma_{l_{\alpha}}^{-1}\sigma_{l_{\beta}}\in H(\mathbb{I})$ is requested.

\smallskip

$(\Leftarrow)$ Suppose that there exists $\sigma\in H(\mathbb{I})$ such that $\alpha=\sigma\beta$. Then
\begin{equation*}
\alpha^{-1}\alpha=(\sigma\beta)^{-1}\sigma\beta=\beta^{-1}\sigma^{-1}\sigma\beta=\beta^{-1}\mathbb{I}\beta=\beta^{-1}\beta,
\end{equation*}
and hence $\alpha\mathscr{L}\beta$ in $\mathbf{I}\mathbb{N}_{\infty}^n$.

\smallskip

The proof of statement $(ii)$ is similar to $(i)$.

\smallskip

Statement $(iii)$ follows from  $(i)$ and $(ii)$.

\smallskip

$(iv)$ $(\Rightarrow)$ Suppose that $\alpha\mathscr{D}\beta$ in $\mathbf{I}\mathbb{N}_{\infty}^n$. Then there exists $\gamma\in \mathbf{I}\mathbb{N}_{\infty}^n$ such that $\alpha\mathscr{L}\gamma$ and $\gamma\mathscr{R}\beta$ and by statements $(i)$ and $(ii)$ we get that there exist $\sigma_1,\sigma_2\in H(\mathbb{I})$ such that $\alpha=\sigma_1\gamma$ and $\gamma=\beta\sigma_2$, and hence $\alpha=\sigma_1\beta\sigma_2$.

\smallskip

$(\Leftarrow)$ Suppose that there exist $\sigma_1,\sigma_2\in H(\mathbb{I})$ such that $\alpha=\sigma_1\beta\sigma_2$. Then $\alpha=\sigma_1\gamma$ for $\gamma=\beta\sigma_2$, which implies by statements $(i)$ and $(ii)$ that $\alpha\mathscr{L}\gamma$ and $\gamma\mathscr{R}\beta$, and hence $\alpha\mathscr{D}\beta$ in $\mathbf{I}\mathbb{N}_{\infty}^n$.

\smallskip

$(v)$ By Proposition~3.2.17 of  \cite{Lawson-1998} it is sufficient to show that every $\mathscr{D}$-class $D$ in $\mathbf{I}\mathbb{N}_{\infty}^n$ contains a minimal element with the respect to the partial natural order $\preccurlyeq$ on $\mathbf{I}\mathbb{N}_{\infty}^n$. Fix an arbitrary $\mathscr{D}$-class $D$ in $\mathbf{I}\mathbb{N}_{\infty}^n$ and any $\alpha\in D$. Then by Lemma~\ref{lemma-3.5} there exists the unique element $\sigma_{\alpha}$ of the group of units $H(\mathbb{I})$ and the unique idempotent $\varepsilon_{l(\alpha)}$ of $\mathbf{I}\mathbb{N}_{\infty}^n$ such that $\alpha=\sigma_{\alpha}\varepsilon_{l(\alpha)}$. Then $\sigma_{\alpha}^{-1}\alpha=\sigma_{\alpha}^{-1}\sigma_{\alpha}\varepsilon_{l(\alpha)}=\mathbb{I}\varepsilon_{l(\alpha)}=\varepsilon_{l(\alpha)}$ and $\varepsilon_{l(\alpha)}\in D$ by statement $(iv)$. Also statement $(iv)$ and Theorem~\ref{theorem-3.2} imply that if an idempotent $\varepsilon$ belongs to $D$ then $\left|\mathbb{N}^n\setminus\operatorname{dom}\varepsilon\right|=\left|\mathbb{N}^n\setminus\operatorname{dom}\varepsilon_{l(\alpha)}\right|$, and hence all idempotents in $\mathscr{D}$-class $D$ are incomparable with respect to the natural partial order $\preccurlyeq$ on $\mathbf{I}\mathbb{N}_{\infty}^n$. Thus, every  $\mathscr{D}$-class $D$ in $\mathbf{I}\mathbb{N}_{\infty}^n$ contains a minimal element.

\smallskip

$(vi)$ The first statement follows from  $(v)$, $(iv)$ and Theorem~\ref{theorem-3.2}.

Let $\alpha$ and $\beta$ be elements of a some $\mathscr{J}$-class $j$ in $\mathbf{I}\mathbb{N}_{\infty}^n$ such that $\alpha\preccurlyeq\beta$. Then by Lemma~\ref{lemma-3.5} there exist the unique elements $\sigma_{\alpha},\sigma_{\beta}\in H(\mathbb{I})$ and $\varepsilon_{l(\alpha)}, \varepsilon_{l(\beta)}\in E(\mathbf{I}\mathbb{N}_{\infty}^n)$ such that $\alpha=\sigma_{\alpha}\varepsilon_{l(\alpha)}$ and $\beta=\sigma_{\beta}\varepsilon_{l(\beta)}$. By Theorem~\ref{theorem-3.6}, $\sigma_{\alpha}=\sigma_{\beta}$ and $\varepsilon_{l(\alpha)}\preccurlyeq \varepsilon_{l(\beta)}$ in $E(\mathbf{I}\mathbb{N}_{\infty}^n)$. Then $\sigma_{\alpha}^{-1}\alpha=\sigma_{\alpha}^{-1}\sigma_{\alpha}\varepsilon_{l(\alpha)}=\varepsilon_{l(\alpha)}$ and $\sigma_{\alpha}^{-1}\beta=\sigma_{\alpha}^{-1}\sigma_{\beta}\varepsilon_{l(\beta)}=\varepsilon_{l(\beta)}$ are $\mathscr{J}$-equivalent idempotents in $\mathbf{I}\mathbb{N}_{\infty}^n$. By the proof of statement $(v)$ we get that $\varepsilon_{l(\alpha)}=\varepsilon_{l(\beta)}$, and hence by Lemma~\ref{lemma-3.5}, $\alpha=\beta$.
\end{proof}

\section{The structure theorem for the semigroup $\mathbf{I}\mathbb{N}_{\infty}^n$}

Recall \cite{McFadden-O'Carroll-1971}, an inverse semigroup $S$ is said to be \emph{$F$-inverse} if every element of $S$ has a unique maximal element above it in the natural partial order $\preccurlyeq$, i.e. every $\mathfrak{C}_{\textsf{mg}}$-class has a maximum element. It is obvious that every $F$-inverse semigroup contains a unit. An inverse semigroup $S$ is \emph{$E$-unitary} if, whenever $e$ is an idempotent in $S$, $s\in S$ and $e\preccurlyeq s$, then $s$ is an idempotent in $S$ \cite{Saito-1965}.

\smallskip

Lemma~\ref{lemma-3.5} implies the following two corollaries.

\begin{corollary}\label{corollary-4.1}
$\mathbf{I}\mathbb{N}_{\infty}^n$ is an $F$-inverse semigroup for any positive integer $n\geqslant 2$.
\end{corollary}

\begin{corollary}\label{corollary-4.1a}
$\mathbf{I}\mathbb{N}_{\infty}^n$ is an $E$-unitary inverse semigroup for any positive integer $n\geqslant 2$.
\end{corollary}

For any element $s$ of an inverse semigroup $S$ we denote ${\downarrow}s=\{x\in S\colon x\preccurlyeq s\}$, where $\preccurlyeq$ is the natural partial order on $S$.

\smallskip

Let $S$ be any $F$-inverse semigroup. Then for any $s\in S$ we denote by $e_s$ the idempotent $ss^{-1}\in S$ and by $t_s$ the
maximum element in the $\mathfrak{C}_{\textsf{mg}}$-class of $s\in S$, and by $T_S$ the set $\{t_s\colon s\in S\}$. We note that $T$ is the set of maximal elements of $S$, and that $T_S$ need not be closed under multiplication \cite{McFadden-O'Carroll-1971}. For each  $x\in S$, as in \cite{McFadden-O'Carroll-1971} let $e_x$ denote the idempotent $xx^{-1}$.

\smallskip

The structure of $F$-inverse semigroups is described in \cite{McFadden-O'Carroll-1971}, and later we need the following two statements for the description of the structure of the semigroup $\mathbf{I}\mathbb{N}_{\infty}^n$.

\begin{lemma}[{\cite[Lemma~3]{McFadden-O'Carroll-1971}}]\label{lemma-4.2}
Let $S$ be an $F$-inverse semigroup and $1_S$  denote the identity
element of $S$. Then
\begin{itemize}
  \item[$(i)$] $E(S)$ is a semilattice with $1_S$ for an identity element.

  \item[$(ii)$] On $T_S$ define the following multiplication:
\begin{equation*}
    u\ast v=t_{uv}, \qquad \hbox{for} \quad u,v\in T_S.
\end{equation*}
  Then $(T_S,*)$ is a group with $1_S$ for an identity element, and each $t\in T_S$ has $t^{-1}$ as its group inverse in $(T_S,*)$.
  \item[$(iii)$] For each $t\in T_S$ the map $\mathfrak{F}_t\colon E(S)\rightarrow {\downarrow}e_t$, defined by
\begin{equation*}
    (f)\mathfrak{F}_t=tft^{-1}, \qquad f\in E(S),
\end{equation*}
   is a homomorphism onto ${\downarrow}e_t$, and moreover $\mathfrak{F}_{1_S}$  is the identity map on $E(S)$.

   \item[$(iv)$] For each $t\in T_S$ $(1_S)\mathfrak{F}_t=e_t$ and $(e_t)\mathfrak{F}_{t^{-1}}=e_{t^{-1}}$.
   \item[$(v)$] For each $u,v\in T_S$
\begin{equation*}
   \left((1_S)\mathfrak{F}_u\right)\mathfrak{F}_v\cdot (f)\mathfrak{F}_{u\ast v}=\left((f)\mathfrak{F}_u\right)\mathfrak{F}_v, \qquad  \hbox{for all } \quad f\in E(S);
\end{equation*}
   \item[$(vi)$] If $u,v\in T_S$ then
\begin{equation*}
    f\cdot(g)\mathfrak{F}_u\leqslant e_{u\ast v},
\end{equation*}
   for all idempotents $f\leqslant e_u$ and $g\leqslant e_v$ of $S$.
\end{itemize}
\end{lemma}

\begin{theorem}[{\cite[Theorem~3]{McFadden-O'Carroll-1971}}]\label{theorem-4.3}
Let $S$ be an $F$-inverse semigroup and let $\mathscr{S}=\bigcup_{t\in T_S}\left({\downarrow}e_t\times\{t\}\right)$. Define the multiplication $\circ$ on $\mathscr{S}$ as follows. If $u,v\in T_S$ then for idempotents $f\leqslant e_u$ and $g\leqslant e_v$ put
\begin{equation}\label{eq-circ}
    (f,u)\circ(g,v)=\left(f\cdot(g)\mathfrak{F}_u,u\ast v\right).
\end{equation}
Then  $\circ$  is a well-defined multiplication on $\mathscr{S}$ and $\left(\mathscr{S},\circ\right)$ is isomorphic to $S$ under the map $\mathfrak{H}\colon S\rightarrow \mathscr{S}$, $s\mapsto\left(ss^{-1},t_s\right)$.
\end{theorem}

Let $A$ and $B$ be semigroups, $\operatorname{\textsf{End}}(B)$ be the semigroup of endomorphisms of $B$ and the following homomorphism $\mathfrak{h}\colon A\rightarrow \operatorname{\textsf{End}}(B)\colon b\mapsto \mathfrak{h}_b$ is defined. Then the set $A\times B$ with the semigroup operation
\begin{equation*}
    (a_1,b_1)\cdot(a_2,b_2)=\left(a_1a_2,(b_1)\mathfrak{h}_{a_2}b_2\right), \qquad a_1,a_2\in A, \; b_1,b_2\in B,
\end{equation*}
is called the \emph{semidirect product} of the semigroup $A$ by $B$ with the respect to the homomorphism $\mathfrak{h}$ and it is denoted by $A\ltimes_\mathfrak{h}B$ \cite{Lawson-1998}. In this case we say that the right action of the semigroup $A$ is defined on the semigroup of endomorphisms of $B$. We remark that a semidirect product of two inverse semigroups is not need an inverse semigroup (see \cite[Section~5.3]{Lawson-1998}).

\begin{lemma}\label{lemma-4.4}
The map $\mathfrak{h}\colon H(\mathbb{I})\rightarrow\operatorname{\textsf{End}}\left(E(\mathbf{I}\mathbb{N}_{\infty}^n)\right)$, $\sigma\mapsto\mathfrak{h}_\sigma$, where  $(\alpha)\mathfrak{h}_\sigma=\sigma^{-1}\alpha\sigma$ is the automorphism of the semilattice $E(\mathbf{I}\mathbb{N}_{\infty}^n)$, is a homomorphism, and moreover $\mathfrak{h}_\mathbb{I}$ is the identity automorphism of $E(\mathbf{I}\mathbb{N}_{\infty}^n)$.
\end{lemma}

\begin{proof}
For any $\sigma\in H(\mathbb{I})$, $\varepsilon,\iota\in E(\mathbf{I}\mathbb{N}_{\infty}^n)$ we have that
\begin{equation*}
    (\varepsilon\iota)\mathfrak{h}_\sigma=\sigma^{-1}\varepsilon\iota\sigma=\sigma^{-1}\varepsilon\sigma\sigma^{-1}\iota\sigma= (\varepsilon)\mathfrak{h}_\sigma(\iota)\mathfrak{h}_\sigma,
\end{equation*}
and hence $\mathfrak{h}_\sigma$ is an endomorphism of the semilattice $E(\mathbf{I}\mathbb{N}_{\infty}^n)$. Also, since for any $\sigma\in H(\mathbb{I})$ and $\varepsilon\in E(\mathbf{I}\mathbb{N}_{\infty}^n)$ the element $\sigma\varepsilon\sigma^{-1}$ is an idempotent of the semigroup $\mathbf{I}\mathbb{N}_{\infty}^n$ and $(\sigma\varepsilon\sigma^{-1})\mathfrak{h}_\sigma=\sigma^{-1}(\sigma\varepsilon\sigma^{-1})\sigma=\varepsilon$, the homomorphism $\mathfrak{h}_\sigma$ is a surjective map. It is obvious that $\mathfrak{h}_\mathbb{I}$ is the identity automorphism of the semilattice $E(\mathbf{I}\mathbb{N}_{\infty}^n)$.

\smallskip

Suppose that $(\varepsilon)\mathfrak{h}_\sigma=(\iota)\mathfrak{h}_\sigma$ for some $\gamma\in H(\mathbb{I})$ and $\varepsilon,\iota\in E(\mathbf{I}\mathbb{N}_{\infty}^n)$. Since $H(\mathbb{I})$ is the group of units of the semigroup $\mathbf{I}\mathbb{N}_{\infty}^n$, the following equalities
\begin{equation*}
    \sigma^{-1}\varepsilon\sigma=(\varepsilon)\mathfrak{h}_\sigma=(\iota)\mathfrak{h}_\sigma=\sigma^{-1}\iota\sigma
\end{equation*}
imply that
\begin{equation*}
    \varepsilon=1\varepsilon 1=\sigma\sigma^{-1}\varepsilon\sigma\sigma^{-1}=\sigma\sigma^{-1}\iota\sigma\sigma^{-1}=1\iota 1=\iota,
\end{equation*}
and hence $\mathfrak{h}_\sigma$ is an automorphism of the semilattice $E(\mathbf{I}\mathbb{N}_{\infty}^n)$.

\smallskip

Fix arbitrary $\sigma_1,\sigma_2\in H(\mathbb{I})$. Then for any idempotent $\varepsilon\in \mathbf{I}\mathbb{N}_{\infty}^n$ we have that
\begin{equation*}
    (\varepsilon)\mathfrak{h}_{\sigma_1\sigma_2}= (\sigma_1\sigma_2)^{-1}\varepsilon\sigma_1\sigma_2= \sigma_2^{-1}\sigma_1^{-1}\varepsilon\sigma_1\sigma_2= \sigma_2^{-1}(\varepsilon)\mathfrak{h}_{\sigma_1}\sigma_2= \left((\varepsilon)\mathfrak{h}_{\sigma_1}\right)\mathfrak{h}_{\sigma_2}= (\varepsilon)\left(\mathfrak{h}_{\sigma_1}\cdot\mathfrak{h}_{\sigma_2}\right),
\end{equation*}
and hence so defined map $\mathfrak{h}\colon H(\mathbb{I})\rightarrow\operatorname{\textsf{End}}\left(E(\mathbf{I}\mathbb{N}_{\infty}^n)\right)$ is a homomorphism.
\end{proof}

The following theorem describes the structure of the semigroup $\mathbf{I}\mathbb{N}_{\infty}^n$.

\begin{theorem}\label{theorem-4.5}
Let $n$ be any positive integer $\geqslant 2$. Then the semigroup $\mathbf{I}\mathbb{N}_{\infty}^n$ is isomorphic to the semidirect product ${\mathscr{S}_n\ltimes_\mathfrak{h}(\mathscr{P}_{\infty}(\mathbb{N}^n),\cup)}$ of free semilattice with the unit  $(\mathscr{P}_{\infty}(\mathbb{N}^n),\cup)$ by the symmetric group $\mathscr{S}_n$.
\end{theorem}

\begin{proof}
Since by Theorem~\ref{theorem-3.2} the group of units $H(\mathbb{I})$ of the semigroup $\mathbf{I}\mathbb{N}_{\infty}^n$ is isomorphic to the symmetric group $\mathscr{S}_n$, it is sufficient to show that the semigroup $\mathbf{I}\mathbb{N}_{\infty}^n$ is isomorphic to the semidirect product $H(\mathbb{I})\ltimes_\mathfrak{h}E(\mathbf{I}\mathbb{N}_{\infty}^n)$ the semilattice $E(\mathbf{I}\mathbb{N}_{\infty}^n)$ by the group of units $H(\mathbb{I})$ of $\mathbf{I}\mathbb{N}_{\infty}^n$ with the respect to the homomorphism $\mathfrak{h}\colon H(\mathbb{I})\rightarrow\operatorname{\textsf{End}}\left(E(\mathbf{I}\mathbb{N}_{\infty}^n)\right)$, $\sigma\mapsto\mathfrak{h}_\sigma$, where $(\varepsilon)\mathfrak{h}_\sigma=\sigma^{-1}\varepsilon\sigma$, $\varepsilon\in E(\mathbf{I}\mathbb{N}_{\infty}^n)$.

\smallskip

We define a map $\mathfrak{T}\colon \mathbf{I}\mathbb{N}_{\infty}^n \to H(\mathbb{I})\ltimes_\mathfrak{h} E(\mathbf{I}\mathbb{N}_{\infty}^n)$ in the following way
\begin{equation*}
    (\alpha)\mathfrak{T}=\left(\sigma_\alpha,\alpha^{-1}\alpha\right),
\end{equation*}
where the element $\sigma_\alpha$ of the group of units $H(\mathbb{I})$ of $\mathbf{I}\mathbb{N}_{\infty}^n$ is defined by formula $\alpha\mapsto \sigma_\alpha$ in Theorem~\ref{theorem-3.7}. By Lemma~\ref{lemma-3.5} and Theorem~\ref{theorem-3.7} the map $\mathfrak{T}\colon \mathbf{I}\mathbb{N}_{\infty}^n \to H(\mathbb{I})\ltimes_\mathfrak{h} E(\mathbf{I}\mathbb{N}_{\infty}^n)$ is well defined and by Theorem~\ref{theorem-3.8}$(iv)$ it is surjective. Suppose there exist $\alpha,\beta\in \mathbf{I}\mathbb{N}_{\infty}^n$ such that $(\alpha)\mathfrak{T}=(\beta)\mathfrak{T}$. Then $\left(\sigma_\alpha,\alpha^{-1}\alpha\right)=\left(\sigma_\beta,\beta^{-1}\beta\right)$ and by Thejrem~\ref{theorem-3.6} we get that
\begin{equation*}
    \alpha=\sigma_\alpha\alpha^{-1}\alpha=\sigma_\beta\beta^{-1}\beta=\beta,
\end{equation*}
and hence the map $\mathfrak{T}\colon \mathbf{I}\mathbb{N}_{\infty}^n \to H(\mathbb{I})\ltimes_\mathfrak{h} E(\mathbf{I}\mathbb{N}_{\infty}^n)$ is injective.

\smallskip

Fix arbitrary $\alpha,\beta\in \mathbf{I}\mathbb{N}_{\infty}^n$. Then Lemma~\ref{lemma-3.5} and Theorem~\ref{theorem-3.7} imply that $\alpha\beta\mathfrak{C}_{\textsf{mg}}\gamma_\alpha\gamma_\beta$, and since $\sigma_\alpha,\sigma_\beta\in H(\mathbb{I})$, we get that $\sigma_\alpha\sigma_\beta=\sigma_{\alpha\beta}$. This implies that
\begin{equation*}
    (\alpha)\mathfrak{T}(\beta)\mathfrak{T}= \left(\sigma_\alpha,\alpha^{-1}\alpha\right)\left(\sigma_\beta,\beta^{-1}\beta\right)= \left(\sigma_\alpha\sigma_\beta,\sigma_\beta^{-1}\alpha^{-1}\alpha\sigma_\beta\beta^{-1}\beta\right)= \left(\sigma_{\alpha\beta},\sigma_\beta^{-1}\alpha^{-1}\alpha\sigma_\beta\beta^{-1}\beta\right).
\end{equation*}
By Lemma~\ref{lemma-4.4} the map $\mathfrak{h}_\sigma\colon E(\mathbf{I}\mathbb{N}_{\infty}^n)\rightarrow E(\mathbf{I}\mathbb{N}_{\infty}^n)\colon\alpha\mapsto\sigma^{-1}\alpha\sigma$ is an automorphism of the semilattice  $E(\mathbf{I}\mathbb{N}_{\infty}^n)$, and hence we get that $\sigma_\beta^{-1}\alpha^{-1}\alpha\sigma_\beta$ is an idempotent of the semigroup  $\mathbf{I}\mathbb{N}_{\infty}^n$. By Lemma~\ref{lemma-3.5} and Theorem~\ref{theorem-3.6} for any $\alpha\in \mathbf{I}\mathbb{N}_{\infty}^n$ there exists the unique element $\sigma_\alpha$ of the group of units $H(\mathbb{I})$ such that  $\alpha\preccurlyeq\sigma_\alpha$. Lemma~1.4.6 from \cite{Lawson-1998} implies that $\beta=\sigma_\beta\beta^{-1}\beta$, and hence
\begin{equation*}
\begin{split}
  \sigma_\beta^{-1}\alpha^{-1}\alpha\sigma_\beta\beta^{-1}\beta & = \left(\sigma_\beta^{-1}\alpha^{-1}\alpha\sigma_\beta\right)\left(\beta^{-1}\beta\right)\left(\beta^{-1}\beta\right)=\\
    & = \left(\beta^{-1}\beta\gamma_\beta^{-1}\right)\left(\alpha^{-1}\alpha\right)\left(\sigma_\beta\beta^{-1}\beta\right)=\\
    & = \left(\sigma_\beta\beta^{-1}\beta\right)^{-1}\left(\alpha^{-1}\alpha\right)\left(\sigma_\beta\beta^{-1}\beta\right)=\\
    & = \beta^{-1}\left(\alpha^{-1}\alpha\right)\beta=\\
    & = \left(\beta^{-1}\alpha^{-1}\right)\left(\alpha\beta\right)=\\
    & = \left(\alpha\beta\right)^{-1}\left(\alpha\beta\right).
\end{split}
\end{equation*}
Therefore, we have that
\begin{equation*}
    (\alpha\beta)\mathfrak{T}=\left(\sigma_{\alpha\beta},(\alpha\beta)^{-1}\alpha\beta\right)=(\alpha)\mathfrak{T}(\beta)\mathfrak{T},
\end{equation*}
and hence the map $\mathfrak{T}\colon \mathbf{I}\mathbb{N}_{\infty}^n \to H(\mathbb{I})\ltimes_\mathfrak{h} E(\mathbf{I}\mathbb{N}_{\infty}^n)$ is a homomorphism, which completes the proof of the theorem.
\end{proof}

\begin{remark}
The monoid  $\mathbf{I}\mathbb{N}_{\infty}$ of all partial cofinite isometries of the set of positive integers $\mathbb{N}$ is a submonoid of the monoid $\mathscr{I}_{\infty}^{\!\nearrow}(\mathbb{N})$ of cofinite, monotone, non-decreasing, injective partial transformations of $\mathbb{N}$ \cite{Gutik-Savchuk-2019}. The structure of the semigroup is described in \cite{Gutik-Repovs-2011}. Also, Lemma~\ref{lemma-3.5} implies that for any positive integer $n\geqslant 2$ the semigroup $\mathbf{I}\mathbb{N}_{\infty}^n$ is an inverse submonoid of the monoid $\mathscr{P\!O}\!_{\infty}(\mathbb{N}^n_{\leqslant})$ of monotone injective partial selfmaps of $\mathbb{N}^{n}_{\leqslant}$ with the order product having cofinite domain and image \cite{Gutik-Krokhmalna-2019, Gutik-Pozdniakova-2016, Gutik-Pozdniakova-2016a}.
\end{remark}
\section*{Acknowledgements}

The authors acknowledge Alex Ravsky for his comments and
suggestions.

\end{document}